\documentclass[12pt, reqno]{amsart}

\usepackage{amsmath}
\usepackage{amssymb}
\usepackage{latexsym}

\newenvironment{proof*}{\vskip 2mm\noindent {}}{\hfill $\Box$ \vskip 2mm}

\renewcommand{\Bbb}{\mathbb}
\newcommand{\B}{{\Bbb  B}}
\newcommand{\C}{{\Bbb  C}}
\newcommand{\D}{{\Bbb D}}
\newcommand{\Cn}{{\Bbb  C\sp n}}

\newcommand{\N}{{\Bbb  N}}

\newcommand{\F}{{\mathcal F}}

\newcommand{\I}{{\mathcal I}}
\newcommand{\J}{{\mathcal{J}}}

\newcommand{\m}{{\mathfrak m}}

\newcommand{\PSH}{{\operatorname{PSH}}}

\renewcommand{\dim}{{\operatorname{dim}}}

\newcommand{\eps}{\varepsilon}
\newcommand{\Om}{\Omega}

\newtheorem{theorem}{Theorem}[section]

\newtheorem{lemma}[theorem]{Lemma}
\newtheorem{exam}[theorem]{Example}
\newtheorem{prop}[theorem]{Proposition}
\newtheorem{defn}[theorem]{Definition}

\begin{document}

\title[Powers of ideals and convergence of Green functions]{Powers of ideals and convergence of Green functions with colliding poles}

\author{ Alexander Rashkovskii and Pascal J. Thomas}

\address{Alexander Rashkovskii\\
Faculty of Science and Technology,
University of Stavanger\\
N-4036 Stavanger, Norway}
\email{alexander.rashkovskii@uis.no}

\address{Pascal J. Thomas\\
Universit\'e de Toulouse\\ UPS, INSA, UT1, UTM \\
Institut de Math\'e\-matiques de Toulouse\\
F-31062 Toulouse, France} \email{pascal.thomas@math.univ-toulouse.fr}

\begin{abstract}
Let $\Omega$ be a bounded
hyperconvex domain in $\mathbb C^n$,
and $\I_\varepsilon$ be a family of ideals of holomorphic functions on $\Om$ vanishing at $N$ distinct points
all tending to $a\in\Om$ as
$\varepsilon\to0$.
As is known, convergence of the ideals $\I_\varepsilon$ to an ideal $\I$ does not guarantee the convergence of the pluricomplex Green functions
$G_{\I_\eps}$ to $G_\I$; moreover, the existence of the limit of the Green functions was unclear.
Assuming that all the powers $\I_\varepsilon^p$ converge to some ideals $\mathcal I_{(p)}$,
we prove that the functions $G_{\I_\eps}$
converge, locally uniformly away from $a$, to a function which is essentially the upper envelope of the scaled Green functions $p^{-1} G_{\I_{(p)}}$, $p\in\N$. As examples, we consider ideals generated by hyperplane sections of a holomorphic curve in $\C^{n+1}$ near a singular point. In particular, our result explains the asymptotics for $3$-point models from \cite{MRST}.
\end{abstract}

\keywords{pluricomplex Green function, ideals of holomorphic functions, Hilbert-Samuel multiplicity, flat families. \\ 2010 {\it Mathematics Subject Classification} 32U35, 32A27}

\thanks{Part of this work was carried out during the stay of
the second-named author in the University of Stavanger in October 2009 and
 during the stay of the first-named
author in the Universit\'e Paul Sabatier in October 2011.}

\maketitle

\section{Introduction}

Pluricomplex Green functions are
fundamental solutions of the (complex) Monge-Amp\`ere operator with zero boundary values \cite{Lem}.
Since the operator is non-linear, superposition does not work and it makes sense
to consider pluricomplex Green functions with multiple poles \cite{Lel}.
We will always do this in the framework of a bounded hyperconvex domain
$\Omega \subset \C^n$.

It is well known that a multipole Green function depends continuously on its poles, provided they do not collide. The convergence problem for the Green
functions with simple logarithmic poles at finitely many points as the
poles tend to the origin was considered in \cite{MRST}.
Let $S_\eps:=\{ a_1(\eps), \dots,  a_N(\eps)\} \subset \Omega$ be
our pole set.  Assume that $\lim_{\eps\to0} a_j(\eps) = a \in \Omega$ for all $j$.
The key to the analysis in \cite{MRST} was to consider
$\I_\eps := \{ f \in \mathcal O (\Omega): f(a_j(\eps))=0, 1\le j \le N\}$,
the radical ideal associated to $S_\eps$, and its limit, taken
in an appropriate sense (see Section \ref{basic} for a precise definition).
There it was proved that, if the respective limits existed,
with the limits of Green functions taken in $L^1_{\rm loc}(\Omega)$, then
$$
\lim_{\eps\to0}G_{\I_\eps}\ge G_{\lim\I_\eps}.
$$
Furthermore, we have $\lim_{\eps\to0}G_{\I_\eps}= G_{\lim\I_\eps}$
if and only if the ideal $\lim\I_\eps$ is a
complete intersection, that is to say, admits exactly $n$ generators
(and in that case convergence of the Green functions always does occur,
and is uniform on compacta of $\Omega \setminus \{a\}$).
A detailed study of the example of families of three points showed
that in cases where the limit ideal is not a complete intersection,
the limits of the Green functions may exist, even though they are far
from being the Green function of the limit ideal.

\medskip

The complete intersection condition is equivalent to the fact that
the codimension (or ``length") of the ideal equals its Hilbert-Samuel
multiplicity, which is defined asymptotically from the lengths of the powers of the ideal.
The main idea of the present paper is to take all powers of $\I_\eps$
before passing to the limit as $\eps\to 0$, and use the infinite
family of those limits to determine the limit of the Green functions.

Our main result is the following:
\begin{theorem}
\label{main}
Let $\{\I_\eps\}_{\eps\in A}$ be a family of ideals of holomorphic functions vanishing at distinct points $a_1(\eps),\ldots, a_N(\eps)$ of a bounded hyperconvex domain $\Om\subset\Cn$, where $A$ is a set in the complex plane, $0\in\overline A\setminus A$. Assume that all $a_j\to a\in\Om$ and $\I^p_\eps\to\I_{(p)}$ for all $p\in\N$ as $\eps\to 0$ along $A$.Then the limit of the Green functions $G_{\I_\eps}$ exists and equals essentially the upper envelope of the scaled Green functions of the limit ideals:
$$\lim_{\eps\to0}G_{\I_\eps}(z)=\limsup_{y\to z}\sup_{p\in\N}\, p^{-1} G_{\I_{(p)}}(y).$$
\end{theorem}

Observe that we have used the subset $A$ in order to allow convergence along any
partial set of parameters (subsequences for instance).

An  ingredient in our proof which should be of interest in itself
is Theorem \ref{bootstrap}, which proves that for families of pluricomplex
Green functions with a fixed number of poles, all reasonable notions of convergence
coincide (the weak convergence in local integrability implies the strong one,
uniform on compacta).

We also provide some examples to show how the limits of Green functions for
three points investigated in \cite{MRST} and even \cite{DQHT} can be obtained much faster,
 and some of those results can be generalized to the higher-dimensional case. We also obtain results in the case of sections of holomorphic curves.

\medskip

{\bf Acknowledgements.} The second named author would like to thank Norman Levenberg
and Nguyen Quang Dieu for stimulating discussions on the topic of this article.

\section{Some basic notions}
\label{basic}

Let ${\mathcal O}(\Om)$ be the space of all holomorphic functions on a bounded hyperconvex domain  $\Om\subset\Cn$. Given an ideal $\I\subset{\mathcal O}(\Om)$, $V(\I)$ denotes its zero variety:
$$V(\I)=\{z\in\Om:\: f(z)=0,\ \forall f\in\I\}.$$

In what follows, we always assume $V(\I)$ to be a finite set. Recall that the \emph{length} of such an ideal is $\ell (\mathcal I)= \dim\, \mathcal O/\mathcal I<\infty$, and the \emph{Hilbert-Samuel multiplicity} is $$ e(\mathcal I)=\lim_{k\to\infty}\frac{n!}{k^n}\,\ell(\mathcal I^k)<\infty.$$
It is known that $e(\I)\ge \ell(\I)$, and the two values are equal if and only if $\I$ is a complete intersection ideal, which means that it has precisely $n$ generators {\cite[Ch. VIII, Theorem 23]{Za-Sa}}.

\medskip

Let $0\in {\overline A}\setminus A\subset \C$ and let
$(\I_\eps)_{\eps \in A}$ be a family of finite length ideals in ${\mathcal O}(\Omega)$. {\it Convergence} of such ideals we will understand in the topology of the Douady space \cite{Dou1}. In particular, it implies
\begin{equation}\label{eq:limitlength}
\ell(\lim_{\eps\to0}\I_\eps)=\lim_{\eps\to0}\ell(\I_\eps).
\end{equation}
As was shown in \cite{MRST}, this convergence is equivalent to the one given in the definition below.

\begin{defn}{\rm \cite{MRST}}
\label{convid}
\begin{enumerate}
\item[(i)]
$\liminf\limits_{A\ni\eps \to 0}\I_\eps$ is the ideal consisting of all $f\in {\mathcal O}(\Omega)$ such that $f_\eps\to f$ locally
uniformly on $\Omega$, as $\eps\to 0$, where $f_\eps \in \I_\eps$.\\
\item[(ii)]
$\limsup\limits_{A\ni\eps \to 0}\I_\eps$ is the ideal
of ${\mathcal O}(\Omega)$ generated by all functions $f$ such that
$f_j\to f$ locally uniformly, as $j\to \infty$, for some
sequence $\eps_j\to 0$ in $A$ and $f_j\in \I_{\eps_j}$.\\
\item[(iii)]
If the two limits are equal, we say that the family
$\I_\eps$ {\rm converges} and write
$\lim\limits_{A\ni\eps \to 0}\I_\eps$ for
the common value of the upper and lower limits.
\end{enumerate}
\end{defn}
If it is clear from
the context which set $A$ we are referring to, then we just drop it from the subscript.

\medskip

The main object of the note is the pluricomplex Green function for an ideal $\I$ of ${\mathcal O}(\Om)$, defined as follows.
\medskip
\begin{defn}{\rm\cite{RaSi}}
\label{greenideal}
For each $a\in \Omega$, let $(\psi_{a,i})_i$ be
a (local) system of generators of $\I$. Then the {\rm Green function} of $\I$ is
$G_{\I}(z) = \sup\{u(z):\: u\in\F_I\}$, where
\begin{equation}\label{eq:FI}\F_\I = \sup \big\{ u(z):\: u \in PSH_-(\Omega), \
u(z) \le \max_i \log |\psi_{a,i}| + O(1) , \forall a \in \Omega \big\}.\end{equation}
\end{defn}
It was proved in \cite{RaSi} that the function $G_\I$ belongs to the class $\F_\I$ and, moreover,
\begin{equation} \label{asympgreen}
G_\I(z)=\max_i \log |\psi_{a,i}| + O(1).
\end{equation}
In addition, it satisfies $(dd^cG_\I)^n=0$ on $\Om\setminus V(\I)$ and, if $V(I)\Subset\Om$, it equals $0$ on the boundary of $\Om$. Furthermore, it is the only plurisubharmonic function with these properties.
This implies, in particular, that for every power $\I^p$ of $\I$,
\begin{equation}\label{eq:grpower}
G_{\I^p}=p\,G_\I.
\end{equation}
Note also that, in our setting of finite length ideals on bounded pseudoconvex domain, one can always choose {\sl global} generators $\psi_i\in{\mathcal O}(\Om)$ and, when $V(\I)=\{a\}$, relation (\ref{asympgreen}) implies that the residual Monge-Amp\`ere mass of $G_\I$ at $a$ equals that of the function $\frac12\log\sum|\psi_i|^2$, so by \cite[Lemma~2.1]{D8},
\begin{equation}\label{eq:grfmass}
(dd^c G_\I)^n=e(\I)\delta_a.
\end{equation}

\medskip

A closely related (though technical) object is the {\it greenification} of a plurisubharmonic function near its singularity point.

\begin{defn}{\rm\cite{R7}}
\label{greenific} Given a function $\varphi\in PSH^-(\Om)$, its {\rm greenification} at a point $a\in\Om$ is the upper regularization $g_\varphi$ of the function $\sup\{u\in PSH^-(\Om):\: u\le \varphi+O(1)\ {\rm near\ } a\}$.
\end{defn}

The function $g_\varphi$ is maximal on $\Om\setminus\{\varphi=-\infty\}$. If $\varphi$ is locally bounded near the boundary of $\Om$, then $g_\varphi=0$ on $\partial \Om$.  Obviously, $\varphi\le g_\varphi$. Furthermore, $\varphi= g_\varphi+O(1)$ near $a$ if $\varphi$ is locally bounded and maximal on a punctured neighborhood of $a$, and in this case it coincides with the {\it Green function for the singularity} $\varphi$ introduced in \cite{Za0}, see also \cite{Za}. Note that the relation $(dd^c g_\varphi)^n(a)=(dd^c\varphi)^n(a)$ remains true without the maximality assumption on $\varphi$.

\section{Modes of convergence}
\label{conv}

\begin{theorem}
\label{bootstrap}
Let $S_\eps = \{ a_1^\eps, \dots, a_N^\eps \}$ with $\lim_{\eps\to0} a_j^\eps=a$,
$1 \le j \le N$. Suppose that $\lim_{\eps\to0} G_{S_\eps } = g$ in
$L^1_{loc}(\Omega \setminus \{a\})$.
Then the convergence takes place uniformly on compacta of
$\Omega \setminus \{a\}$, and
$(dd^c g)^n = N \delta_a$; in particular, $g$ is maximal plurisubharmonic
on $\Omega \setminus \{a\}$.
\end{theorem}

Note that we may assume $a=0$ without loss of generality.
We will use the well-known rough estimates of a multipole Green function:
$$
\min_{a\in S} G_a \ge G_S \ge \sum_{a\in S} G_a.
$$

The proof of the Theorem rests on the proof of the analogous
fact in the special case of the ball, Proposition \ref{bootball} below,
and will be given at the end of this section.

In what follows, $\|\cdot\|$ stands for the usual Euclidean norm, $a\cdot b=\sum a_jb_j$, $B(a,r)=\{ z\in \mathbb C^n : \|z-a\|<r\}$, and $\B^n=B(0,1)$.

\begin{lemma}
\label{equicont}
Let $K \subset \mathbb B^n \setminus \{0\}$ be a compact set. Then for any $\eta >0$,
there exists $\delta >0$ depending only on $\eta$, $N$ and $K$ such that if $z_1,
z_2 \in K$, $\| z_1-z_2 \| \le \delta$ and
$S\subset B(0,\delta)$, then $|G_S(z_1)-G_S(z_2)| < \eta$.
\end{lemma}

\begin{proof}
Since the roles of $z_1$ and $z_2$ are symmetric, it will be enough to show that
$$
G_S(z_1) \ge G_S(z_2)-\eta
$$
whenever $z_2$ is close enough to $z_1$.

\begin{lemma}
\label{map}
For any $\eta_1 >0$, one can find $\delta_1 >0$ depending only on
$\eta_1$, $\Omega$, $N$ and $K$ such that if $z_1,
z_2 \in K$, $\| z_1-z_2 \| \le \delta_1$ and $S\subset B(0,\delta_1)$,
then there exists a holomorphic map $\Phi$ defined on $\mathbb B^n$
such that $\Phi|_S = id|_S$, $\Phi(z_1)=z_2$ and $\Phi(\mathbb B^n) \subset B(0,1+\eta_1)$.
\end{lemma}
\begin{proof}
Let
$$
P(z):= \prod_{a \in S} (z-a) \cdot \frac{\bar z_1}{\|z_1\|},
\Phi(z) := z + \frac{P(z)}{P(z_1)} (z_2-z_1).
$$
Take any $\delta_1 \le \frac12 \min_K \|z\|$. Then $|P(z_1)| \ge 2^{-N} \|z_1\|^N
\ge 2^{-N} \min_K \|z\|$.  On the other hand, $|P(z)| \le 2^N$ for $z\in\mathbb B^n$.
So the conclusion will hold whenever $\| z_1-z_2 \| \le 2^{-2N} \|z_1\|^N \eta_1 =:\delta_1$.
\end{proof}

We need to construct a function plurisubharmonic and negative on $B(0,1+\eta_1)$
that is a competitor for $G_S$.  First recall that
$$
G_S (z) \ge \sum_{a\in S} G_{a}(z) = \sum_{a\in S} \log \| \phi_a(z)\|,
$$
where $\phi_a$ is an automorphism of the unit ball exchanging $a$ and $0$.
From an explicit formula for $\phi_a$ \cite{Ru}, we know that
$$
1 - ||\phi_a(z)||^2 = \frac{ (1-||a||^2)(1-||z||^2)} {| 1-z \cdot \bar a |^2}.
$$
Since $| 1-z \cdot \bar a |^2 \ge (1-||a||)^2$,
we deduce, if $\|a\| \le \delta \le \frac13$,
$$
\| \phi_a (z) \|^2 \ge
1- \frac{1+\|a\|}{1-\|a\|}(1-\|z\|^2)
\ge 1- 2 (1-\|z\|^2)
.
$$
We may assume
$1-\|z\|^2 \le \frac14$
so
we have $\| \phi_a (z) \|^2 \ge \frac12$ and
$$
\log \|\phi_a (z) \| \ge  -(1-\| \phi_a (z) \|^2) \ge - 2(1-\|z\|^2)
\ge - 4(1-\|z\|) \ge 4\log \|z\|.
$$
Assume that $\eta_1<\frac14$. Let
\begin{eqnarray*}
v(z) &=&  G_S (z),   \mbox{ for }
\|z\|
\le e^{-2\eta_1},\\
&=& \max\left( G_S (z), \eta_1 + \log \|z\| + 4 N  \log \|z\| \right),
 \mbox{ for } e^{-2\eta_1} \le  \|z\|\le 1, \\
&=& \eta_1 + \log \|z\| +  4 N \log \|z\|, \mbox{ for } 1 \le  \|z\| < 1+\eta_1.
\end{eqnarray*}
Since $G_S (z) > \eta_1 + \log \|z\| + 4 N  \log \|z\| $
for $\|z\| = e^{-2\eta_1}$ and  $G_S (z) =0 < \eta_1$ for $\|z\|=1$, we have
 $v \in PSH (B(0,1+\eta_1))$. Let $v_1:= v  - (1+ 4 N)\log (1+\eta_1) - \eta_1
 \in PSH_-(B(0,1+\eta_1))$.

 Clearly $v_1 \circ \Phi \in PSH_-(\mathbb B^n)$.
 Since $v_1 = G_S + O(1)$ in a fixed ball containing $S$, and $\Phi$ fixes $S$,
 $v_1 \circ \Phi \le G_S$.  We apply this at the point $z_1$:
 $$
 v_1 (z_2) = v_1 \circ \Phi (z_1) \le G_S (z_1).
 $$
 Now we choose $\eta_1$ small enough so that $K \subset B(0, e^{-2\eta_1} )$,
 and so
 that
 $$
 v_1(z_2)= G_S(z_2) - (1+ 2 N)\log (1+\eta_1) - \eta_1 > G_S(z_2) - \eta
 $$
 for  $\|a\| \le \delta (\eta, \eta_1,N) <\delta_1$.
\end{proof}

\begin{prop}
\label{bootball}
Let $S_\eps = \{ a_1^\eps, \dots, a_N^\eps \}\subset\B^n$ with $\lim_{\eps\to0} a_j^\eps=0$,
$1 \le j \le N$. Suppose that $\lim_{\eps\to0} G_{S_\eps } = g$ in
$L^1_{loc}(\mathbb B^n \setminus \{0\})$.
Then the convergence takes place uniformly on compacta of
$\mathbb B^n \setminus \{0\}$; in particular, $g$ is maximal plurisubharmonic
on $\mathbb B^n \setminus \{0\}$.
\end{prop}
\begin{proof}
Since the topology of uniform convergence on compacta is metrizable, it will be
enough to show that any subsequence $\{G_{S_{\eps_j}}\}$
admits a convergent subsequence.  First consider
a fixed compact set $K \subset \mathbb B^n \setminus \{0\}$.  Then $\{G_{S_{\eps_j}}\}$
converges in $L^1(K)$, so there exists a subsequence, which we denote by $\{G_j\}$,
which converges almost everywhere on $K$.
Note that the standard rough estimates on Green functions show that
all $G_{S_\eps}$ are bounded by common bounds on $K$, and therefore
so is $g$ (where it is defined).

We want to show that the subsequence $\{G_j\}$
satisfies the uniform Cauchy criterion.  Let $\delta >0$.
Let $\eta_0:= \min \left( \min_K \|z\|, 1-\max_K \|z\| \right)$.
By Lemma \ref{equicont}
applied to $\{z : dist(z,K)\le \eta_0/2\}$,
there exists $\eta = \eta(\delta) \le \eta_0/2$ and $J_1=J_1(\delta)$ such that
for any $j\ge J_1$, the oscillation of
$G_j$ on any ball of radius $\eta$ is at most $\delta/4$.

We cover the compact set
$K$ by balls $B(c_k, \eta), 1\le k \le m(\delta)$. The almost everywhere convergence
implies that for each $k$, there exists $c'_k \in B(c_k, \eta)$ such
that $\lim_{j\to\infty} G_j (c'_k) = g(c'_k) $. Since there is only
a finite number of $c'_k$, there exists $J_2\ge J_1$ such that for
any $j\ge J_2$ and any $k\le m(\delta)$, $|G_j(c'_k)-g(c'_k)|\le \delta/4$.
The rest is routine:
 for any $j,l \ge J_2$, for any $z \in K$, we choose $k$ such
that $z\in B(c_k, \eta)$ and we have
\begin{eqnarray*}
\left| G_l(z) - G_j(z) \right|
&\le&
\left| G_l(z) - G_l(c'_k) \right| + \left|  G_l(c'_k) -  g(c'_k) \right|\\
&+& \left|  g(c'_k) - G_j(c'_k) \right|
 + \left|  G_j(c'_k) - G_j(z) \right|
\le \delta.
\end{eqnarray*}
To get the uniform convergence on any compact set, we repeat this
argument over an exhaustion sequence $\{K_m\}$ of compacta,
and perform a diagonal extraction.
\end{proof}

\begin{proof*}{\it Proof of Theorem \ref{bootstrap}.}
To prove the uniform convergence,
by \cite[Lemma 4.5]{MRST}, it is enough to show that there exists $g$
and $\delta_0 >0$ such that for any $\delta < \delta_0$,
$| G^\Omega_{S_\eps} - g | \le C$ for $\|z\|=\delta$ and $|\eps|<\eps(\delta)$.
Of course we take $\delta_0$ small enough so that $B(0,\delta_0) \Subset \Omega$.
Then it is easy to see that on $B(0,\delta_0)$,
$| G^\Omega_{S_\eps} - G^{B(0,\delta_0)}_{S_\eps}| \le C(\delta_0, \Omega)$.
By Proposition~\ref{bootball}, $\{G^{B(0,\delta_0)}_{S_\eps}\}$
converges uniformly on compacta of $B(0,\delta_0) \setminus \{0\}$,
in particular on the sphere of radius $\delta$, and we have the desired property.

To prove the fact about Monge-Amp\`ere measure, first notice that
uniform convergence on compacta (or even pointwise convergence) clearly
implies, using the maximum principle, that $g$ is maximal plurisubharmonic
on $\Omega \setminus\{a\}$, and so $(dd^c)^ng = C \delta_a$, with $C\le N$
by \cite[Proposition 1.13]{MRST}.  To finish the proof, we only need to show
that $(dd^c)^ng (\Omega)=N$.

For any non-positive function $u$ on $\Omega$ and $m\in \N^*$, let $T_m (u):=
\max(-m,u)$.
Then, the rough estimates imply that $\{g\le -m\} \Subset \Omega$ and
$\{G^\Omega_{S_\eps}\le -m\} \Subset \Omega$, so
$$(dd^c)^n T_m(g) (\Omega)=(dd^c)^ng (\Omega),\quad (dd^c)^n T_m(G^\Omega_{S_\eps}) (\Omega)=(dd^c)^n G^\Omega_{S_\eps} (\Omega)=N.$$

Again by the rough estimates and the fact that all the $a_j(\eps)$
tend to $a$, for any fixed $m$ there exists $r_m>0, \eps_m>0$ such that
$B(a,r_m) \subset \{g\le -m-1\}$ and $B(a,r_m) \subset \{G^\Omega_{S_\eps}\le -m-1\}$,
for $|\eps|\le \eps_m$. For any compactum $K\subset \Omega$, $K\setminus B(a,r_m)$
is a compactum of $\Omega \setminus \{a\}$, so $G^\Omega_{S_\eps}$ converges uniformly
to $g$ on it, therefore $T_m(G^\Omega_{S_\eps})$ converges uniformly to $T_m(g)$
on $K$. This implies that
$$(dd^c)^n T_m(g) (\Omega) = \lim_{\eps\to 0} (dd^c)^n T_m(G^\Omega_{S_\eps}) (\Omega)=N.$$

Note that convergence of the Monge-Amp\`ere measures can also be proved by
noticing that uniform convergence on compacta of $\Omega \setminus \{a\}$
implies convergence in capacity, and that type of convergence guarantees convergence
of the corresponding Monge-Amp\`ere measures, see \cite{Ce} or \cite{PHH}.
\end{proof*}

\section{Proof of the main result}
\label{proofmain}

Let $0\in\Om$ and let $\I_\eps$ be a family of finite length ideals in ${\mathcal O}(\Omega)$ such that $V(\I_\eps)\to \{0\}$ as $\eps\to 0$. We assume that all the powers $\I_\eps^p$ converge (in the sense of Definition~\ref{convid}) to some limits $\I_{(p)}$, $p=1,2,\ldots$. Surely, $V(\I_{(p)})=\{0\}$.

\medskip

The crucial point is the following simple observation.

\begin{prop}\label{lem:graded} For any $p,q\in\N$,
\begin{equation}\label{eq:graded}\I_{(p)}\cdot\I_{(q)}\subset\I_{(p+q)}.\end{equation}
\end{prop}

\begin{proof} If $f\in \I_{(p)}$ and $g\in\I_{(q)}$, then they are limits of certain functions $f_\eps\in\I_\eps^p$ and $g_\eps\in\I_\eps^q$, respectively. Note that $f_\eps\,g_\eps\in\I_\eps^{p+q}$. Therefore, $$fg=\lim_{\eps\to 0} f_\eps\,g_\eps\in\liminf_{\eps\to0}\I_\eps^{p+q}=\I_{(p+q)}.$$
\end{proof}

Note that the inclusion in (\ref{eq:graded}) can be strict.

\begin{exam}\label{exam:3point}  $3$-point model in $\C^2$. \end{exam}
\noindent Let $a_1(\eps)=(\eps,0)$, $a_2(\eps)=(0,\eps)$, $a_3(\eps)=(0,0)$.
The ideals $$\I_\eps=\langle z_1z_2, z_1(z_1-\eps),z_2(z_2-\eps)\rangle$$ converge to the
ideal $\I=\langle z_1^2,z_1z_2,z_2^2\rangle$, while the squares $\I_\eps^2$ converge to the ideal $\I_{(2)}$ generated by $\I^2$ and the function $z_1z_2(z_1+z_2)$. Indeed, the ideal $$\I_\eps^2=\langle z_1^2z_2^2, z_1^2(z_1-\eps)^2, z_2^2(z_2-\eps)^2, z_1^2z_2(z_1-\eps), z_1z_2^2(z_2-\eps), z_1z_2(z_1-\eps)(z_2-\eps)\rangle $$
contains the function $z_1z_2(z_1+z_2)-\eps z_1z_2=\eps^{-1}[z_1^2z_2^2 - z_1z_2(z_1-\eps)(z_2-\eps)]$.\hfill$\square$

\medskip

Relation (\ref{eq:graded}) means precisely that $\{\I_{(p)}\}$ is a {\it graded family of ideals}. In particular, this implies that we have control over the Hilbert-Samuel multiplicities of these limit ideals:

\begin{prop}\label{prop:volume} There exists the limit
$$
e(\I_\bullet):=\lim_{p\to\infty}p^{-n}e(\I_{(p)})= \inf_{p\in\N}p^{-n}e(\I_{(p)})=\limsup_{p\to\infty}n!p^{-n}\ell(\I_{(p)}).
$$
\end{prop}

\begin{proof} This follows directly from Theorem~1.7 of \cite{Mustata} valid for any graded family of zero dimensional ideals. \end{proof}

The value $e(\I_\bullet)$ is called the {\it volume} of the graded family $\I_\bullet$.

\medskip

Proposition \ref{lem:graded} implies
\begin{equation}\label{eq:gradedgreen}
G_{\I_{(p)}\cdot\I_{(q)}}\le G_{\I_{(p+q)}},
\end{equation}
and we are going to deduce from this a convergence result for $G_{\I_{(p)}}$ -- more precisely, for the \emph{scaled Green functions}
\begin{equation}\label{eq:Ghat} \widehat G_{\I_{(p)}}=p^{-1}G_{\I_{(p)}}.\end{equation}

\begin{prop}\label{prop: conv} There exists the limit
\begin{equation}\label{eq:lim}V(z)=\lim_{p\to\infty} \widehat G_{\I_{(p)}}(z)=\sup_{p\in\N} \widehat G_{\I_{(p)}}(z)\end{equation} whose upper regularization $G_{\I_\bullet}(z)=\limsup_{x\to z}V(x)$ is a plurisubharmonic function satisfying
\begin{equation}\label{eq:mass} (dd^c  G_{\I_\bullet})^n=e(\I_\bullet)\delta_0.\end{equation}
Furthermore, $\widehat G_{\I_{(p)}}\to G_{\I_\bullet}$ in $L^p(\Omega)$ for all $p\in [1,n]$.
\end{prop}

\begin{proof}
By (\ref{asympgreen}), since the product of ideals is generated by pairwise products of their generators, we have
$$G_{\I_{(p)}\cdot\I_{(q)}}=G_{\I_{(p)}}+G_{\I_{(q)}}+O(1),$$
so the function $G_{\I_{(p)}}+G_{\I_{(q)}}$ belongs to the class $\F_{\I_{(p)}\cdot\I_{(q)}}$ defined in (\ref{eq:FI}), and inequality (\ref{eq:gradedgreen}) gives us
\begin{equation}\label{eq:conv} G_{\I_{(p)}}+G_{\I_{(q)}}\le G_{\I_{(p+q)}}.\end{equation}

Relations (\ref{eq:lim}) follow now from (\ref{eq:conv}) by standard arguments; see, for example, \cite[Lemma~1.4]{Mustata} applied to $\alpha_p=-G_{\I_{(p)}}(z)\ge0$.

Now we turn to proving (\ref{eq:mass}). By the Choquet lemma, one can find a sequence $\widehat G_{\I_{(p_j)}}$ increasing almost everywhere to the function $G_{\I_\bullet}$; actually, one can just choose $\widehat G_{\I_{(p!)}}$, cf.~\cite{R10}. Indeed, relation (\ref{eq:graded}) implies, in particular,
$ \I_{(p)}^k\subset\I_{(kp)}$,
 and using (\ref{eq:grpower}) we get
$$k\,G_{\I_{(p)}}=G_{\I_{(p)}^k}\le G_{\I_{(kp)}}$$
for any $k,p\in\N$.
Therefore,
$ \widehat G_{\I_{(p)}}\le \widehat G_{\I_{(kp)}}$, so
$\widehat G_{\I_{(p!)}}\ge \widehat G_{\I_{(q)}}$ for all $q\le p$.

By the monotone convergence theorem for the complex Monge-Amp\`ere operator,
$$(dd^c\widehat G_{\I_{(p!)}})^n\to (dd^c G_{\I_\bullet})^n.$$
Since $(dd^c\widehat G_{\I_{(p!)}})^n=(p!)^{-n}e(\I_{(p!)})$, relation (\ref{eq:mass}) follows from the first equality in Proposition~\ref{prop:volume}.

Finally, since  $\widehat G_{\I_{(p)}}\le G_{\I_\bullet}$ for any $p$, the last assertion follows from Lemma~\ref{lem:ACCP} below and H\"{o}lder's inequality.
\end{proof}

\medskip

\begin{lemma}\label{lem:ACCP} Let $u,v\in \PSH^-(\Om)$ be maximal on $\Om\setminus\{x\}$, equal to $0$ on $\partial \Om$, and $u\le v$ in $\Om$. Then
$$\int_\Om(v-u)^n (dd^cw)^n\le n!\,\int_\Om w\left[(dd^cu)^n-(dd^cv)^n\right]$$
for any $w\in\PSH(\Om)$, $0\le w\le 1$.
\end{lemma}

\begin{proof} This is a particular case of \rm\cite[Prop. 3.4]{NP}.
\end{proof}

\medskip

Next, we get a lower bound for the limit of Green functions. In order to state it without assumption on uniform convergence of the Green functions $G_{\I_\eps}$, we will use here the notion of greenification of a plurisubharmonic function, see Definition~\ref{greenific}.

\begin{prop}\label{prop:lbound} Let $\varphi$ be the largest plurisubharmonic minorant of the function $\liminf_{\eps\to 0}G_{\I_\eps}$. Then its greenification $g_\varphi$ satisfies
$g_\varphi\ge G_{\I_\bullet}$.
Consequently, if $G_{\I_\eps}$ converge to $\varphi$ locally uniformly outside $0$, then $\varphi\ge G_{\I_\bullet}$.
\end{prop}

\begin{proof}
By \cite[Proposition~1.5]{MRST},
$$\varphi\ge G_{\liminf\limits_{\eps\to 0}\J_\eps}+O(1)$$ for any
family of zero-dimensional ideals $\J_e$ such that $V(\liminf\limits_{\eps\to 0}\J_\eps)=\{0\}$.
 Applying this to $\J_\eps=\I_\eps^p$ and taking into account (\ref{eq:grpower}) and (\ref{eq:Ghat}), we get the inequalities
$$ \varphi\ge \widehat G_{\I_{(p)}}+C_p$$
with some constants $C_p$. By passing to the greenifications, we deduce
$g_\varphi\ge \widehat G_{\I_{(p)}}$ and then, in view of Proposition~\ref{prop: conv}, $g_\varphi\ge G_{\I_\bullet}$. If, in addition, the convergence of  $G_{\I_\eps}$ to $\varphi$ is locally uniform outside $0$, then $\varphi=g_\varphi$, which completes the proof.
\end{proof}

\medskip
 From now on, we assume that
the ideals $\I_\eps$ are intersections of maximal ideals. In this case we can compute the volume  $e(\I_\bullet)$ of the graded family  $\I_\bullet$.

\begin{prop}
Let $\I_\eps$ be radical ideals with  $V(\I_\eps)$ consisting of $N$ different points $a_1(\eps), \ldots, a_N(\eps)$ for all $\eps\neq 0 $ sufficiently small. Then $e(\I_\bullet)=N$.
\end{prop}

\begin{proof}
Since the length is stable under limit transitions, $\ell(\I_{(p)})=\ell(\I_\eps^p)$, so Proposition~\ref{prop:volume} gives us
$$ e(\I_\bullet)=\limsup_{p\to\infty}n!\,p^{-n}\ell(\I_\eps^p).$$

Each ideal $\I_\eps^p$ consists of all functions $f\in{\mathcal O}(\Om)$ satisfying
$$\frac{\partial^{|\beta|} f}{\partial z^{\beta}}(a_j(\eps))=0,\quad |\beta|<p,\  1\le j\le N,$$
so $\ell(\I_\eps^p)=\binom{p+n-1}{n}N$ and
$$ e(\I_\bullet)=\lim_{p\to\infty}n!\,p^{-n}\binom{p+n-1}{n}N=N.$$
\end{proof}

\medskip

Now we are ready to prove our main result. It will rest on the following domination principle.

\begin{lemma}\label{lemma:dom} {\rm \cite[Lemma 6.3]{R7}} Let $u_1$ and $v_2$ be two plurisubharmonic solutions of the Dirichlet problem
$(dd^cu)^n=\delta_0$, $u|_{\partial\Om}=0$. If $u_1\ge u_2$ on $\Om$, then $u_1\equiv u_2$.
\end{lemma}
{\it Remark.} This can actually be deduced from the more advanced Lemma~\ref{lem:ACCP}. A more general version of the domination principle can be found in \cite{ACCP} as well.

\begin{prop}
\label{oldmain}
{\sl Let $\{\I_\eps\}_{\eps\in A}$ be a family of ideals of holomorphic functions vanishing at distinct points $a_1(\eps),\ldots, a_N(\eps)$ of a bounded hyperconvex domain $\Om\subset\Cn$, where $A$ is a set in the complex plane, $0\in\overline A\setminus A$. Assume that all $a_j\to a\in\Om$ and $\I^p_\eps\to\I_{(p)}$ for all $p\in\N$ as $\eps\to 0$ along $A$. If the limit of the Green functions $G_{\I_\eps}$ for the ideals $\I_\eps$ exists, uniformly on compact subsets of $\Om\setminus\{a\}$, then }
$$\lim_{\eps\to0}G_{\I_\eps}(z)=\limsup_{y\to z}\sup_{p\in\N}\, p^{-1} G_{\I_{(p)}}(y).$$
\end{prop}

\begin{proof} Denote the limit of the Green functions $G_{\I_\eps}$ by $\varphi$.
The uniform convergence implies $(dd^c \varphi)^n=\lim (dd^c G_{\I_\eps})^n=N\delta_0$. By
 Proposition~\ref{prop:lbound}, we have $g\ge G_{\I_\bullet}$, and (\ref{eq:mass}) implies $g= G_{\I_\bullet}$ by Lemma~\ref{lemma:dom}.
\end{proof}

%
%
%

\begin{proof*}{\it Proof of Theorem \ref{main}.}
First notice that the set $\{G_{S_\eps}, \eps \in A\}$ is sequentially weakly compact
in the dual of the space of bounded continuous functions on $\Omega$, since
$0\ge G_{S_\eps} \ge \sum_{j=1}^N G_{a_j(\eps)}$, and each of those functions has
a uniformly bounded $L^1$ norm when $a_j(\eps)$ is close to $0$. From this standard arguments of measure theory can be used to show that a subsequence converging in $L^1_{\rm loc}$ can always be extracted.  Or we can simply use \cite[Theorem 3.2.12, p. 149]{Hor},
noticing that the case of a subsequence converging to $-\infty$ is excluded by the estimate
above.

Now suppose that $\I^p_\eps\to\I_{(p)}$ for all $p\in\N$, and suppose to get a contradiction that $G_{S_\eps}$ is not converging uniformly on compacta to
$G_{\I_\bullet}$.  Then
we can get $\{\eps_j\} \subset A$, $\eps_j \to 0$, such that
$\sup_K |G_{\eps_j}-G_{\I_\bullet}| \ge \delta>0$ for $j$ large enough and for some
compactum $K\subset \Omega \setminus \{0\}$.  Then there is a subsequence
of $\{\eps_j\}$, which we denote again by $\{\eps_j\}$, such that $G_{\eps_j}$
converges in $L^1_{\rm loc}$, and by Theorem \ref{bootstrap} uniformly on compacta. By Proposition \ref{oldmain}, it must converge to $G_{\I_\bullet}$: a contradiction.
\end{proof*}

Finally, we mention the following finiteness result.

\begin{theorem}\label{theo:fin} If, in addition to the conditions of Theorem \ref{main} (or Proposition \ref{oldmain}), the Hilbert-Samuel multiplicity of some limit ideal $\I_{(p)}$ equals $p^nN$, then the limit of the Green functions equals
$\widehat G_{\I_{(p)}}$.
\end{theorem}

\begin{proof} The condition on multiplicity of $\I_{(p)}$ means that $(dd^c \widehat G_{\I_{(p)}})^n(0)=N$. Since $\widehat G_{\I_{(p)}}\le G_{\I_\bullet}$ and the latter one has the same Monge-Amp\`ere mass at $0$, Lemma~\ref{lemma:dom} gives us $\widehat G_{\I_{(p)}}= G_{\I_\bullet}$.
\end{proof}

In particular, this works when the limit ideal $G_{\I_{(1)}}$ is a complete intersection, because in this case $e(\I_{(1)})=N$. More advanced situations will be considered in the next section.

\section{Examples and questions}
There are several natural examples where our theorem works. In what follows, $\m_a\subset{\mathcal O}(\Om)$ denotes the maximal ideal composed by the functions vanishing at the point $a\in\Om$.

\begin{exam} Two points in $\Cn$.\end{exam}
Let $a_1$, $a_2$ be continuous mappings of the unit disk $\D$ into the unit polydisk $\D^n$ such that $a_1(\eps)\neq a_2(\eps)$ for all $\eps\in\D\setminus\{0\}$, and $a_i(0)=0$. Set $\I_\eps=\m_{a_1(\eps)}\cap\m_{a_2(\eps)}$.

The family  $\{\I_\eps\}_{\eps\neq 0}$ need not have a limit as $\eps\to0$.
By compactness, there is a sequence $\eps_k\to0$ such that
$[a_1(\eps_k)- a_2(\eps_k)]\to \nu \in \mathbb P^{n-1} \C$,
where $[z]$ denotes the class of $z$ in $\mathbb P^{n-1} \C$, for $z\in \D^n\setminus \{0\}$. As is easy to see, the sequence $\I_{\eps_k}$ has a limit, $\I_{(1)}$, whose multiplicity equals $2$ because it is a complete intersection.

According to Section~6.1 of \cite{MRST},
the corresponding Green functions $G_{\I_{\eps_k}}$ converge to a function whose Monge-Amp\`ere mass at $0$ equals $2$. Therefore, by Theorems~\ref{main} and \ref{theo:fin}, the limit function coincides with $G_{\I_{(1)}}$.\hfill$\square$

\medskip

\begin{exam} Complete intersection case of the $4$-point ideals in $\C^2$. \end{exam}
\noindent Consider the ideals $\I_\eps=\m_0\cap\m_{(\eps,0)}\cap\m_{(0,\eps)}\cap\m_{(\eps,\eps)}$ in a bounded hyperconvex domain $\Omega\subset \C^2$ containing the origin.
They converge to $\I_{(1)}=\langle z_1^2,z_2^2\rangle$ whose Hilbert--Samuel multiplicity equals $4$, so the limit of the Green functions is the Green function of $\I_{(1)}$. The existence of the limit is however quite a simple fact in this case, see \cite{MRST}. (Note that a much stronger result was proved there. Namely, when the limit ideal $\I_{(1)}$ is a complete intersection, then the limit of the Green functions exists and coincides with the Green function of $\I_{(1)}$.)\hfill$\square$
\medskip

\begin{exam}\label{exam:3}  $3$-point problem in $\C^2$.\end{exam}
\noindent This more complicated problem is also treated in \cite{MRST}, where the following two cases were considered: the generic one modeled by the ideals
$\I_\eps=\m_0\cap\m_{(\eps,0)}\cap\m_{(0,\eps)}$,
and the degenerate one modeled by
$\I_\eps=\m_0\cap\m_{(\rho(\eps),0)}\cap\m_{(0,\eps)}$
with $\rho(\eps)/\eps\to0$. Both families converge to $\m_0^2$, however the limits of the corresponding Green functions were shown to be different.

In the first case, as was mentioned in Example~\ref{exam:3point}, the squares $\I_\eps^2$ converge to the ideal $\I_{(2)}$ generated by $\m_0^4$ and the function $z_1z_2(z_1+z_2)$. Note that
$$\widehat G_{\I_{(2)}}= \frac12\, G_{\I_{(2)}}=
\max\{2\log|z_1|,2\log|z_2|,\frac12\log|z_1z_2(z_1+z_2)|\}+O(1).$$
Comparing it with results from \cite{MRST}, we see that this is precisely the asymptotic of the limit function $\lim_{\eps\to0} G_{\I_\eps}$ and, therefore, in this case one has $G_{\I_\bullet}=\widehat G_{\I_{(2)}}$.

Another way to check this, according to Theorem~\ref{theo:fin}, is to show that the Monge-Amp\`ere mass of $\widehat G_{\I_{(2)}}$   equals $3$. This can be done easily, since the mass of $G_{\I_{(2)}}$ can be computed as the Hilbert--Samuel multiplicity of the ideal $\I_{(2)}$, and the latter equals the multiplicity of generic mappings $(f_1,f_2)$ for $f_1,f_2\in {\I_{(2)}}$, which is $12$.

\medskip

In the second case, the limit ideal $\I_{(2)}$ is monomial and contains, besides $\m_0^4$, the function
$z_1^2 z_2$. The Hilbert--Samuel multiplicity of $\I_{(2)}$ easily computes to be $12$, so the Monge-Amp\`ere mass of $\widehat G_{\I_{(2)}}$  equals $3$ again. And the limit of the Green functions is indeed, up to a bounded term, $\frac12 \log \max(|z_1^4|, |z_1^2 z_2|, |z_2^4|)$.

Note that in \cite{MRST} the convergence of the Green functions were established by using a sophisticated machinery of constructing special analytic disks, while now we get the results almost for free.

\medskip
The case not treated in \cite{MRST} was that when all the poles tend to $0$ along
the same asymptotic directions, or to be more precise, when one point is equal to $(0,0)$
(which is no loss of generality) and the other two verify
$\lim_\eps a_2^\eps/\|a_2^\eps\| =a_3^\eps/\|a_3^\eps\|=v$.  This question
is dealt with in \cite{DQHT}, but there is no answer there for
the limit of the Green functions when the limit ideal is not a complete intersection.

Recall that the situation may be reduced to
$a_1^\varepsilon = (0,0)$, $a_2^\varepsilon = (\varepsilon,0)$,
$a_3^\varepsilon = (\rho, \delta\rho)$,
with $\rho=\rho(\varepsilon)$, $\delta=\delta(\varepsilon)$,
 $|\eps|^2 \ge |\eps-\rho|^2 +  |\rho|^2  |\delta|^2 \ge  |\rho|^2 +  |\rho|^2  |\delta|^2$,
 so in particular $|\eps-\rho| \ge \frac12 |\eps|$.

 Denote $\alpha=\frac{\delta}{\rho-\eps}$. It is shown in \cite[Theorem 1.5]{DQHT}
 that $\lim_{\eps\to0}\I_\eps$ is a complete intersection ideal when
 $\lim_{\eps\to0}\alpha^{-1}\neq0$, and that $\lim_{\eps\to0}\I_\eps=\m_0^2$ when
 $\lim_{\eps\to0}\alpha^{-1}=0$. We want to settle the latter case.

From  \cite{DQHT}, there are polynomials  in $\mathcal{I}_\varepsilon$
\begin{equation*}
\begin{aligned}
Q_1^\varepsilon(z) &= z_1^2 - \eps z_1 - \alpha^{-1}{\delta}z_2
\to z_1^2;\\
Q_2^\varepsilon(z) &= z_2 \big(z_1 - \rho\big) \to z_1z_2;\\
Q_3^\varepsilon(z) &= z_2 \big(z_2 - \delta \rho\big) \to z_2^2,
\end{aligned}
\end{equation*}
so $\I_\eps^2$ contains
$$
Q_1^\varepsilon(z)Q_3^\varepsilon(z)-(Q_2^\varepsilon(z))^2
=
- \alpha^{-1}
\left(
z_2^3 + \eps\alpha z_1^2 z_2^2
+  (1-\eps) \delta\rho\alpha z_1^2 z_2
-2 \rho \alpha z_1 z_2^2
+ \eps \rho \alpha  z_2^2
\right),
$$
therefore $z_2^3\in \I_{(2)}$. Clearly, $\m_0^4 \subset \I_{(2)}$.
So $\widehat G_{\I_{(2)}} \ge \max ( 2 \log |z_1|, \frac32 \log |z_2|) +O(1)$. Using the fact
 that the Monge-Amp\`ere mass of this lower bound is $3$, or that the
 multiplicity of the mapping $f(z)=(z_1^4,z_2^3)$ is $12$, we see that in fact, by Theorem~\ref{theo:fin},
 $\lim_\eps G_{S_\eps} = \widehat G_{\I_{(2)}} = \max ( 2 \log |z_1|, \frac32 \log |z_2|) +O(1)$
 (if $\Omega=\D^2$, we don't even need the $O(1)$ term).
 \hfill$\square$

\medskip
\begin{exam}\label{exam:n+1}  Generic $n+1$ points in $\C^n$.\end{exam}
\noindent Consider ideals $\I_\eps$ with $S_\eps=V(\I_\eps)$ consisting of the origin and the points $\eps e_k$ for basis vectors $e_k$, $1\le k\le n$. They are generated by the functions $z_iz_j$ with $i< j$ and $z_i(z_i-\eps)$ and they converge, as expected, to $\m_0^2$.

\begin{prop}
\label{simplex}
The Green functions $G_{S_\eps}$ converge, uniformly on compacta of $\Omega\setminus\{0\}$,
to $w_n (z)+O(1)$, where
$$
w_n (z) :=
\max \left\{ \frac1{\# A}
\log \left| (\sum_{j\in A} z_j)\prod_{j\in A} z_j  \right|, A \subset \{1,\dots,n\}
\right\}  .
$$
\end{prop}

Notice that for $n=2$, this yields the previous result in the generic case of Example~\ref{exam:3}.

\begin{proof}
We will proceed by induction on the dimension. For $n=1$, $w_1(z)=2\log |z|$ and
the result is immediate by direct calculation (we can add up the Green functions for
each pole in this case).

In general, notice that if $L$ is an affine subspace of $\C^n$, and $z=(z',z'')$
where $L=\{(z',0): z'\in \C^k\}$ in suitable coordinates,
$G^\Omega_S (z) \ge G^{\Omega\cap L}_{S\cap L} (z')$, simply because
 the latter function is a competitor in the supremum that defines the former.

We assume the result is true in $\C^d$, for any $d\le n-1$.
 Therefore, to prove that $\liminf_\eps G_{S_\eps} \ge w_n$, it will be enough
 to prove
  \begin{equation}
  \label{topdim}
 \liminf_\eps G_{S_\eps} \ge \log \left| (\sum_{j=1}^n z_j)\prod_{j=1}^n z_j  \right| ,
  \end{equation}
 and to use the induction hypothesis on all the subspaces $\mbox{Span} \{e_j, j\in A\}$
 for $\# A \le n-1$.

 We will look at $S_\eps$ as a simplex in a space with one more dimension,
formally let
$$
\varphi_\eps : (z_1, \dots, z_n) \mapsto (\eps- \sum_j z_j,  z_1, \dots, z_n)\in
\C^{n+1},
$$
where the coordinates in $\C^{n+1}$ are $(z_0, z_1, \dots, z_n)$.
Then $S_\eps = \varphi_\eps^{-1} ( \cup_{j=0}^n \C e_j )$.

The ideal associated to $\cup_{j=0}^n \C e_j$ is
$\mathcal I := < z_i z_j, 0\le i<j\le n >$.

We are using multi-index notation, $\alpha = (\alpha_0, \dots ,\alpha_n) \in \N^{n+1}$.
\begin{lemma}
\label{admissible}
$$
\mathcal I^k= \left< z^\gamma : |\gamma|=2k , \gamma_j \le k, 0\le j \le n
\right>.
$$
\end{lemma}

\begin{proof}
The equality for $k=1$ is the very definition of $\mathcal I$.
By taking products of generators,
it is immediate that
$\mathcal I^k\subset \left< z^\gamma : |\gamma|=2k , \gamma_j \le k, 0\le j \le n
\right>.$ Conversely, suppose the reverse inclusion is verified up to $k-1$.
 For any $\gamma$ involved in the
right hand side, the definition  implies that there are at least two distinct
indices $i,j$ such that $\gamma_i, \gamma_j \ge 1$, and no more than
two indices such that $\gamma_i, \gamma_j \ge k$. Pick $i,j$
so that $\gamma_i, \gamma_j$ are maximal among the exponents.
So $z^\gamma = z_iz_j z^{\gamma'}$,
where $|\gamma'|=2k-2$, and $\gamma'_l \le k-1$ for any $l$:
the induction hypothesis implies that $z^\gamma\in \I \cdot \I^{k-1}$.
\end{proof}

Now we want to prove that $(\sum_{j=1}^n z_j)\prod_{j=1}^n z_j\in \I_{(n)}$,
which will prove \eqref{topdim} by Theorem \ref{main} and the fact that
$\widehat G_{\I_{(n)}} \le G_{\I_\bullet}$.  By Lemma \ref{admissible},
$z_0 \cdots z_n (z_0 + \cdots +z_n)^{n-1} \in \I^n$, so using the pull-back,
$(\eps - (z_1 + \cdots +z_n)) z_1 \cdots z_n \eps^{n-1} \in \I_\eps^n$, so dividing by
$ \eps^{n-1}$ and taking the limit, we have the desired fact.
So we have proved the ``$\ge$" part of our Proposition.  To prove
the reverse inequality, we must show that $w_n$ has small enough
Monge-Amp\`ere mass.

\begin{lemma}
\label{massest}
$
(dd^c)^n w_n (0) \le n+1.
$
\end{lemma}

\begin{proof}
Let $t_k=t/k$, where $t$ is a common multiple of $1,\ldots,n$. Consider the mapping $f$ with components
$$ f_k=\sum_{|A|=k} (\sum_{j\in A} z_j)^{t_k}\prod_{j\in A} z_j^{t_k}, \quad A \subset \{1,\dots,n\},
\ k=1,\ldots,n.$$
First we check that it has an isolated zero at $0$. If $f_n=0$, then either $z_k=0$ for some $k$, or $\sum_{j} z_j=0$. In both cases, the equation $f_{n-1}=0$ gives us then either $z_{l}=0$ for some $l\neq k$, or $\sum_{j\neq k} z_j=0$. Continuing this, we arrive at the last step to the unique solution $z=0$.

Now, since the components of $f$ are homogeneous polynomials of degrees $(k+1)t_k$, Bezout's theorem gives us the multiplicity of $f$ equal to
$$\prod_k (k+1)t_k= \frac{(n+1)!\,t^n}{n!}= (n+1)t^n,$$
which shows that the Hilbert--Samuel multiplicity of the ideal generated by the functions $(\sum_{j\in A} z_j)^{t_k}\prod_{j\in A} z_j^{t_k}$, $A \subset \{1,\dots,n\}$, is at most $(n+1)t^n$, so the Monge-Amp\`ere mass of $w_n$ at $0$ is at most $n+1$.
\end{proof}

Now let $u$ be any  limit point in $L^1_{\rm{loc}}$ of $\{G_{S_\eps}\}$.  By Theorem \ref{bootstrap}, the convergence is
in fact uniform on compacta of $\Omega\setminus\{0\}$, and
$(dd^cu)^n= (n+1)\delta_0$.  We also have $u\ge w_n + O(1)$, so passing to the greenifications,
we have $u=g_u\ge g_{w_n}$ and  $(dd^c)^ng_{w_n}= c\delta_0$,
$c\le n+1$, so that, by Lemma~\ref{lemma:dom}, the two functions must be in fact equal. Now by the usual reasoning
(there is only one possible limit point), the limit of $G_{S_\eps}$ must be $ g_{w_n}
= w_n +O(1)$.
\end{proof}

%
%

\medskip

\begin{exam}\label{exam:sections} Hyperplane sections of holomorphic curves. \end{exam}
\noindent More generally, let us have a holomorphic curve (one-dimensional analytic variety) $\Gamma$ such that $0\in\C^{n+1}$ is its singular point. By Thie's theorem, there exists a neighborhood $U$ of $0$ such that for a choice of coordinates $(z,w)\in\Cn\times\C$, the restriction of $\Gamma$ to $U$ lies in the cone $\{|z|\le C|w|\}$, $C>0$. Let for $\eps\in\C\setminus\{ 0\}$, sufficiently small, $\I_\eps$ be the ideal in $O(\D^n)$ determined by the points $a_k=a_k(\eps)$ such that $(a_k,\eps)\in\Gamma$. By Propositions~III.4.7 and III.4.8 of \cite{Ha}, the collection $\{\I_\eps^p\}_{\eps}$ with any $p\in\N$  has a unique continuation to a flat family, so there exists a limit of $\I_\eps^p$ as $\eps\to0$. Therefore, the limit of the corresponding Green functions $G_{\I_\eps}$ exists and equals the function $G_{\I_\bullet}$. \hfill$\square$

\bigskip

In the end, we would like to mention a few open questions.

\medskip

1. In Example~\ref{exam:n+1}, we have found that $G_{\I_\bullet}= \widehat G_{\I_{(t(n))}}$ with $t(n)$ equal to the least common multiple of $1, \dots, n$ (in particular, $t(n)\le n!$). {\it What is (the asymptotic of) the best possible index $p(n)$ such that $G_{\I_\bullet}= \widehat G_{\I_{(p(n))}}$?}

2. {\it Is it always true that $G_{\I_\bullet}=\widehat G_{\I_{(p)}}$ for some $p\in\N$, at least in the setting of Example~~\ref{exam:sections} ?}

\medskip

3. {\it What can be said in the case of non-radical ideals $\I_\eps$ whose varieties tend to a single point?}


\begin{thebibliography}{11}

\bibitem{ACCP}
{\sc P. {\AA}hag, U. Cegrell, R. Czyz, Ph\d{a}m Ho\`{a}ng Hi\d{\^{e}}p},
{\it Monge-Amp\`ere measures on pluripolar sets}, J. Math. Pures Appl. (9) {\bf 92} (2009), no. 6, 613--627.

\bibitem{Ce}
{\sc U. Cegrell}, {\it  Convergence in capacity}, Canad. Math. Bull. {\bf 55} (2012), no. 2, 242--248.

\bibitem{D8}
{\sc J.-P. Demailly}, {\it Estimates on Monge-Amp\`ere operators derived from a local algebra inequality}, in: Complex Analysis and Digital Geometry. Proceedings from the Kiselmanfest, 2006, Uppsala University, 2009, 131--143.

\bibitem{Dou1}
{\sc A. Douady}, {\it Le probl\`eme des modules pour les sous-espaces analytiques
             compacts d'un espace analytique donn\'e}, Ann. Inst. Fourier (Grenoble) {\bf 16} (1966), fasc. 1,  1--95.

\bibitem{DQHT}
{\sc Duong Quang Hai, P. J. Thomas}, {\it Limit Of Three-Point Green Functions : The Degenerate Case}, arXiv:1205.5899.

\bibitem{Ha}
{\sc R. Hartshorne}, Algebraic Geometry. Springer, 1997.

\bibitem{Hor}
{\sc L. H\"ormander}, Notions of Convexity. Birkh\"auser, Boston, 1994, Progress in Mathematics
no. 127.


\bibitem{Lel} {\sc P. Lelong}, {\it Fonction de Green pluricomplexe et lemmes
de Schwarz dans les espaces de Banach}, J. Math. Pures Appl. {\bf 68} (1989), 319--347.


\bibitem{Lem} {\sc L. Lempert}, {\it La m\'etrique de Kobayashi et la
repr\'esentation des domaines sur la boule}, Bull. Soc. Math. France {\bf 109} (1981), no. 4, 427--474.

\bibitem{MRST}
{\sc J.~Magn\'usson, A. Rashkovskii, R. Sigurdsson, and P. Thomas}, {\it Limits of multipole pluricomplex Green functions},  Int. J. Math. {\bf 23} (2012), no. 6; available at http://arxiv.org/abs/1103.2296

\bibitem{Mustata}
{\sc M. Mustata}, {\it On multiplicities of graded sequences of ideals},  J. Algebra {\bf 256} (2002), 229--249.

\bibitem{NP}
{\sc  Nguyen Van Khue and Ph\d{a}m Ho\`{a}ng Hi\d{\^{e}}p}, {\it A comparison principle for the complex Monge-Ampère operator in Cegrell's classes and applications}, Trans. Amer. Math. Soc. {\bf 361} (2009), no. 10, 5539--5554.

\bibitem{PHH}
{\sc Ph\d{a}m Ho\`{a}ng Hi\d{\^{e}}p}, {\it Convergence in capacity and application}, Math. Scand. {\bf 10}7 (2010), no. 1, 90--102.

\bibitem{R7}
{\sc A. Rashkovskii}, {\it Relative types and extremal problems for
plurisubharmonic functions}, Int. Math. Res. Not. {\bf 2006} (2006),
Article ID 76283, 26 p.

\bibitem{R10}
{\sc A. Rashkovskii}, {\it Asymptotically analytic and other plurisubharmonic singularities}, http://arxiv.org/abs/0910.2137

\bibitem{RaSi}
{\sc A.~Rashkovskii and R.~Sigurdsson}, {\it Green functions with
singularities along complex spaces}, { Int. J. Math.}
{\bf 16} (2005), no.~4, 333--355.

\bibitem{Ru}
{\sc W. Rudin},  Function theory in the unit ball of $\C^n$. Reprint of the 1980 edition. Classics in Mathematics. Springer-Verlag, Berlin, 2008.

\bibitem{Za0}
{\sc V.P. Zahariuta}, {\it Spaces of analytic functions and maximal
plurisubharmonic functions.} D.Sci. Dissertation, Rostov-on-Don,
1984.

\bibitem{Za} {\sc V.P. Zahariuta}, {\it Spaces of analytic functions and
Complex Potential Theory}, Linear Topological Spaces and Complex
Analysis {\bf 1} (1994), 74--146.

\bibitem{Za-Sa} {\sc O. Zariski, P. Samuel}, {\it Commutative algebra. Vol. II},
Reprint of the 1960 edition. Graduate Texts in Mathematics, Vol. 29. Springer-Verlag, New York-Heidelberg, 1975.

\end{thebibliography}
\end{document}